\documentclass{amsart}

\newtheorem{theorem}{Theorem}[section]
\newtheorem{lemma}[theorem]{Lemma}
\newtheorem{proposition}[theorem]{Proposition}
\newtheorem{corollary}[theorem]{Corollary}

\theoremstyle{definition}
\newtheorem{definition}[theorem]{Definition}
\newtheorem{example}[theorem]{Example}

\newtheorem{remark}[theorem]{Remark}

\usepackage{amscd,amssymb}

\begin{document}

\title[Generalizations of almost prime and right $S$-prime ideals in noncommutative rings]
{Generalizations of almost prime and right $S$-prime ideals in noncommutative rings }

\author[Alaa Abouhalaka, {\c S}ehmus F{\i}nd{\i}k, Nico Groenewald]
{Alaa Abouhalaka, {\c S}ehmus F{\i}nd{\i}k, Nico Groenewald}

\address{Department of Mathematics,
\c{C}ukurova University, 01330 Balcal\i,
 Adana, Turkey}
\email{alaa1aclids@gmail.com}
\email{sfindik@cu.edu.tr}

\address{Department of Mathematics, Nelson Mandela University, Gqeberha,
South Africa}
\email{nico.groenewald@mandela.ac.za}

\thanks
{}


\begin{abstract}
Let $R$ be a noncommutative ring, and let $S$ be an  $m$-system of $R$. In this paper, we give more results on the concept of almost prime (right) ideals, that were introduced by the first two authors, especially in (right) $S$-unital rings, local rings, and decomposable rings. In addition, we introduce the concept of almost right $S$-prime ideals, and we show how some findings regarding almost prime ideals can be derived as consequences of almost right $S$-prime ideals. Besides, we show how  almost right $S$-prime ideals behave in related rings such as homomorphic images, quotient rings, and decomposable rings. Finally, we construct  almost right $S$-prime ideals using the Nagata method of idealization. 
 
\end{abstract}

\maketitle

\section{Introduction}

The pioneer study on almost prime ideals in noncommutative rings was done by the first two authors \cite{AbouFindik}, recently.
Let $R$ be a noncommutative ring, and $S$ be an $m$-system of $R$. 
Recall that a proper (right) ideal $P$ of $R$ is called an almost prime (right) ideal, if $AB\subseteq P$ and $AB\not\subseteq P^2$ implies $A\subseteq P$ or $B\subseteq P$ for all (right) ideals $A,B$ of $R$ (see \cite{AbouFindik}).  It gives several equivalent conditions for an  (a right) ideal to be almost prime. Besides, connections between almost prime right ideals and weakly prime right ideals were clarified in the work. Another result in \cite{AbouFindik} gives that the homomorphic image of a fully almost prime right ring (i.e., ring in which every (right) ideal  is an almost prime (right) ideal), is again a fully almost prime right ring. In addition, if $R$ is a fully almost prime right ring so is $R/I$, where $I$ is an ideal of $R$. 
Later, the notion of $S$-prime ideals (where $S$ is an $m$-system of $R$) was studied by the first author \cite{A}, as generalization of $S$-prime ideals (where $S$ is a multiplicatively closed subset of $R$) in commutative rings which were previously introduced in \cite{HK}.  Referring to \cite{A}, an ideal $P$ of  $R$, disjoint from  the $m$-system $S$ is called  a right $S$-prime ideal associated with an element $s\in S$ (briefly a right $S$-prime ideal), if  for all ideals $A$, $B$  of $R$ with $AB \subseteq P$, either $A\langle s\rangle \subseteq P$ or $B\langle s\rangle\subseteq P$.  Theorem 2.9 of \cite{A}, shows that an ideal $P$ is a right $S$-prime ideal if and only if
$(P:\langle s\rangle) = \{r \in R : r\langle s\rangle \subseteq P\}$
is a prime ideal of $R$ for some $s \in S$. In addition,  Theorem 3.11 of \cite {A}, is an $S$-version of Cohen's Theorem. 

Our results in the present article are twofold. We devote Section 2 to (right) almost prime ideals in (right) $S$-unital rings, local rings, and decomposable rings; examination of their fully almost prime (right) rings; and characterization of fully almost prime (right) rings which are not local rings. For instance, Theorem \ref{fully two sided} shows
that if $R$ is a fully almost prime ring which is right \text{\rm S}-unital, and if $M_1$ and $M_2$ are two distinct maximal ideals of $R$, then, $M_1M_2= M_2M_1=M_1\cap M_2$, and it is idempotent. Also, Theorem  \ref{KKKKKK}
states that if $R$ is a ring with unity in which the product of any two maximal right ideals is an almost prime right ideal, which is not an idempotent ideal, then $R$ is a local ring. 
Section 3 is devoted to constracting the notion of almost right $S$-prime ideal. In this context, an ideal $P$ of $R$ disjoint from $S$ is called an almost right $S$-prime ideal associated with an element $s\in S$ (or briefly an almost right $S$-prime ideal) if whenever $A$
and $B$ are ideals of $R$ with $AB\subseteq P$ and $AB\nsubseteq P^{2}$,
implies either $A\langle s\rangle \subseteq P$ or $B\langle s\rangle
\subseteq P$. We show how some findings of almost prime ideals turn as corollaries of almost right $S$-prime ideal. We also give several equivalent definitions for an  ideal to be classified as an almost right $S$-prime ideal. In addition, we study in related rings like homomorphic image, quotient rings, decomposable rings, and finally, we construct almost right $S$-prime ideals using the Method of Idealization, which was introduced by Nagata for commutative rings.

\section{Almost prime right ideals}

Throughout the paper, all rings are associative, noncommutative, and without unity unless stated otherwise; and by
an ideal we mean a proper two sided ideal.

\subsection{S-unital rings}
A ring $R$ is called right (left) S-unital if $a\in aR$ ($a\in Ra$) for all $a \in R$, and  $R$ is called S-unital if it is both right and left S-unital. 

Proposition 2.2 of \cite{AbouFindik} stated that an ideal $P$ of a ring $R$ with unity is an almost prime ideal if and only if $P$ is an almost prime right ideal.
In the next result, we show that this equivalence holds in left S-unital rings.  

\begin{proposition}\label{00}
Let $R$ be a left \text{\rm S}-unital ring and $P$ be an ideal of $R$. Then the following statements are equivalent.

$(1)$ $P$ is an almost prime right ideal.

$(2)$  $P$ is an almost prime  ideal.
\end{proposition}

\begin{proof}
$(1)\Rightarrow(2)$ Trivial by definition.

$(2)\Rightarrow(1)$ Suppose that  $AB\subseteq P$ and $AB\not\subseteq P^2$, for right ideals $A$ and $B$ of $R$.
Since $R$ is left S-unital ring, then $a\in Ra \subseteq RA$ for all $a\in A$. Thus $A \subseteq RA$. Similarly  $B \subseteq RB$. 
Now we have
\[
(RA)(RB) \subseteq RAB \subseteq RP\subseteq P
\]
for ideals $RA$ and $RB$. On the other hand, if $(RA)(RB) \subseteq P^2$, then
\[
AB\subseteq  (RA)(RB) \subseteq P^2,
\]
which is a contradiction. Hence $(RA)(RB) \not\subseteq P^2$, and by $(2)$ we have
either $A\subseteq RA\subseteq P$ or  $B\subseteq RB\subseteq P$.
\end{proof}

The next theorem is an analogue of Lemma 2.7 of \cite{B1} that was proved in integral domains.

\begin{theorem} Let $P$ be an almost prime ideal of a right \text{\rm S}-unital ring $R$, and let $a$ be an element in the center $C(R)$ of $R$ such that $a\not\in P$.
Then, if $a+P$ is a zero divisor in $R/P$, then $aP \subseteq P^2$.
\end{theorem}

\begin{proof} By assumption, there exists $b\in R{\setminus}P$ such that
\[
ab+P=(a+P)(b+P)=P.
\]
Thus, $ab \in P$ and
\begin{align}
(((a\rangle +P)/P)(((b\rangle +P)/P)=&((a\rangle(b\rangle +(a\rangle P+P(b\rangle+P^2+P)/P\nonumber\\
=&((ab\rangle +(a\rangle P+P(b\rangle+P^2+P)/P\subseteq P/P.\nonumber
\end{align}
Hence, $(a\rangle(b\rangle \subseteq P$. Recall that  $(a\rangle\not\subseteq P$, $(b\rangle\not\subseteq P$, and $P$ is an almost prime ideal. Therefore, $(a\rangle(b\rangle \subseteq P^2$.

Now for all $x\in P$, we have $x+(b\rangle\not\subseteq P$. So, $(a\rangle(x+(b\rangle)=(a\rangle x + (a\rangle(b\rangle \subseteq P$.
Again, due of the fact that $P$ is an almost prime ideal, and $(a\rangle\not\subseteq P$, $x+(b\rangle\not\subseteq P$,
we obtain $(a\rangle x\subseteq(a\rangle x + (a\rangle(b\rangle\subseteq P^2$. This implies that $aP\subseteq (a\rangle P \subseteq P^2$.
\end{proof}

\subsection{Local Rings}

Recall that a ring $R$ with unity is said to be local ring if it contains a unique maximal right ideal $M$.      
We will denote it by $(R, M)$. Recall that $M$ is the unique (two sided) maximal ideal of $R$.

\begin{proposition}\label{local2}
Let $(R, M)$ be a local ring, and let $P$ be a right  ideal of $R$ such that $P^2=M^2$. Then $P$ is an almost prime right ideal. 
\end{proposition}

\begin{proof}
Let $A$ and $B$ be right ideals of $R$. Then, $A\subseteq M$ and $B\subseteq M$. Thus, $AB\subseteq M^2 =P^2$, which yields that $P$ is an almost prime right ideal.
\end{proof}

The next two results are consequences of Proposition \ref{local2}.

\begin{proposition}\label{sup}
Let $(R, M)$ be a local ring, and let $P$ be a right  ideal of $R$ such that $M^2\subseteq P$. Then, $P$ is an almost prime right ideal if and only if $M^2 =P^2$. 
\end{proposition}

\begin{proof}
Suppose that $P$ is an almost prime right ideal of $R$.
Then clearly, $P^2\subseteq M^2$ due to the fact that $P\subseteq M$.
If $M^2\subseteq P^2$, then we get that $M^2 =P^2$. Because $M^2\subseteq P$ and $P$ is almost prime,
the case $M^2\not\subseteq P^2$ implies that $M \subseteq P$. This gives that $M=P$, and hence $M^2 =P^2$.
The converse implification is a consequence of Proposition \ref{local2}, which completes the proof.
\end{proof}

The commutative analogue of the above proposition was proved in \cite{B}.

\begin{corollary}\label{c25}
Every local ring $(R, M)$ with $M^2 =0$ is a fully almost prime right ring. Futhermore, it is a fully weakly prime right ring. 
\end{corollary}

\begin{proof}
For any right ideal $P$ of $R$, we have that $P^2=M^2=0$. Thus, we get that $P$ is an almost prime right ideal by Proposition \ref {local2}.
Additionally, $P$ is a weakly prime right ideal by Theorem 2.7 of \cite{AbouFindik}.
\end{proof}

\begin{example}  Let $(R, M)$ be a local ring and $P$ be a right  ideal of $R$ such that $P\cap M^2\subseteq P^2$ ($P\cap M^2=0$).
Then, $P$ is an almost prime right ideal of $R$.
Observe that if $A$ and $B$ are right ideals of $R$ such that $AB \subseteq P$, then $AB \subseteq P\cap M^2\subseteq P^2$ ($AB =0\subseteq P^2)$.
\end{example}

A similar example to the above example was suggested by Victor Camillo (see \cite{B}) for the commutative case.
Now, we collect some propositions to be used in latter results.

\begin{proposition}
Let $R$ be a fully almost prime right ring such that $R=I\oplus J$ for some ideals $I,J$ of $R$. Then $I$ and $J$ are fully almost prime right rings.
\end{proposition}

\begin{proof}
Since $I\cong R/J$ and $J\cong R/I$, the proof follows from Theorem 3.6 of \cite{AbouFindik}.
\end{proof}

\begin{proposition}\label{0}
Let $I$ and $J$ be any ideals of a ring $R$ such that $I\not\subseteq J$ and $J\not\subseteq I$. If $I\cap J$ is an almost prime ideal then $IJ=JI=(I\cap J)^2$. 
\end{proposition}

\begin{proof}
Clearly $(I\cap J)^2\subseteq IJ$ and $IJ\subseteq I\cap J$. Assume that $IJ\not\subseteq (I\cap J)^2$.
Then, either $I\subseteq I\cap J\subseteq J$ or $J\subseteq I\cap J\subseteq I$, since $I\cap J$ is an almost prime ideal.
This contradicts with $I\not\subseteq J$ and $J\not\subseteq I$. Hence $IJ\subseteq (I\cap J)^2$. Similarly one can obtain that $JI=(I\cap J)^2$.
\end{proof}

The next corollary results from Proposition \ref{0}.

\begin{corollary}
Let $R$ be a fully almost prime ring. If $IJ\not\subseteq(I\cap J)^2$ for all ideals $I$ and $J$ of $R$, then the set of ideals of $R$ is totally ordered under inclusion.
\end{corollary}

The following two theorems describe fully almost prime (right) rings which are not local.

\begin{theorem}\label{fully two sided}
Let $R$ be a fully almost prime ring which is right \text{\rm S}-unital.
If $M_1$ and $M_2$ are two distinct maximal ideals of $R$ \text{\rm (}thus R is not a local ring\text{\rm )}, then followings hold.

$(1)$ $M_1M_2= M_2M_1$.

$(2)$ $M_1\cap M_2=M_1M_2$, and it is idempotent.

\end{theorem}

\begin{proof} $(1)$ Since $M_1\not\subseteq M_2,  M_2\not\subseteq M_1$, and $M_1 \cap M_2$ is an almost prime ideal of $R$, then by Proposition \ref{0}
we get that $M_1M_2= (M_1\cap M_2)^2 =M_2M_1$.

$(2)$ We provide the proof of the fact that $M_1\cap M_2=M_1M_2$, which is well known consequence of $(1)$.
Clearly, $M_1M_2 \subseteq M_1 \cap M_2$. 
Because $M_1$ and $M_2$ are distinct maximal ideals, they are comaximal ideals. That is, $M_1+ M_2=R$.
Now let $r\in M_1\cap M_2$. Then $r\in M_1$, $r\in M_2$, and $r\in rR=r(M_1+M_2)$ because $R$ is right S-unital.
Thus, by $(1)$,
\[
r=rm_1+rm_2\in M_2M_1+M_1M_2=M_1M_2
\]
for some $m_1\in M_1$, $m_2\in M_2$.
Therefore $M_1\cap M_2=M_1M_2$, and by Proposition \ref{0}, we get that 
\[
M_1\cap M_2=M_1M_2=(M_1\cap M_2)^2.
\]
\end{proof}

\begin{corollary}
We conclude from Theorem \ref{fully two sided} by induction that if $R$ is a fully almost prime ring which is right \text{\rm S}-unital
with dictinct maximal ideals $M_1,\ldots,M_n$, then,
\[
M_1\cdots M_n=M_{\pi(1)}\cdots M_{\pi(n)},
\]
for every permutation $\pi\in S_n$, $n\geq2$.

\end{corollary}

\begin{theorem}\label{local3}
Let the ring $R$ with unity be a fully almost prime right ring which is not local.
Then, for any two maximal right ideals $M_1$ and $M_2$, either $M_1M_2$ or $M_2M_1$ is idempotent.
\end{theorem}

\begin{proof} Because $R$ is not a local ring, there exist two distinct maximal right ideals $M_1, M_2$ of $R$.
Assume that neither $M_1M_2$ nor $M_2M_1$ is idempotent. Then, $M_1M_2 \not\subseteq(M_1M_2)^2$
gives either $M_1\subseteq M_1M_2$ or $M_2\subseteq M_1M_2$ in the fully almost prime right ring $R$. 
Let $M_2\subseteq M_1M_2$. Then, $M_2\subseteq M_1M_2\subseteq M_1$ which follows that $M_1=M_2$. Contradiction. Thus, $M_1\subseteq M_1M_2\subseteq M_1$.
Therefore, $M_1 =  M_1M_2$. Similarly one can obtain that $M_2=M_2M_1$.
Hence,
\begin{align}
&(M_1)^2=(M_1M_2)M_1=M_1(M_2M_1)=M_1M_2=M_1,\nonumber\\
&(M_2)^2=(M_2M_1)M_2=M_2(M_1M_2)=M_2M_1=M_2.\nonumber
\end{align}
So, $M_1$ and $M_2$ are idempotent right ideals of $R$. As a consequence,
\begin{align}
&(M_1M_2)^2=(M_1M_2)(M_1M_2)=M_1M_1=M_1=M_1M_2,\nonumber\\
&(M_2M_1)^2=(M_2M_1)(M_2M_1)=M_2M_2=M_2=M_2M_2.\nonumber
\end{align}
Hence, both $M_1M_2$ and $M_2M_1$ are idempotent, which contradicts with our assumption.
\end{proof}

The ring $R$ in next example was suggested in Example 2.1 of \cite{AbouFindik}, which provides an illustration for Theorem \ref{local3}.

\begin{example}
All the maximal right ideals of the fully almost prime right ring $R$ are $I$, $J$, $P$.
Clearly $IJ=0$, $JI=J$; $IP=I$, $PI=P$; and $JP=J$, $PJ=0$. Moreover, $0$, $I$ and $P$ are idempotent, but $J$ is not idempotent. 
\end{example}

\begin{theorem}\label{KKKKKK}
Let $R$ be a ring with unity in which the product of any two maximal right ideals is an almost prime right ideal, but not an idempotent ideal.
Then, $R$ is a local ring. 
\end{theorem}

\begin{proof}
Assume that the ring $R$ is not local. Then, there exist two distinct maximal right ideals $M_1, M_2$ of $R$.
Now by assumption, both $M_1M_2$ and $M_2M_1$ are almost prime right ideals, that are not idempotent. 
Then, similar to the proof of Theorem \ref{local3}, we have that $M_1=M_1M_2$ and $M_2=M_2M_1$.
Thus, 
\[
M_1M_2=M_1(M_2M_1)=M_1(M_2(M_1M_2))=(M_1M_2)^2,
\]
that is a contradiction.
\end{proof}

The following theorem is well known, which is necessary in the proof of Theorem \ref{local1}.

\begin{theorem}\label{local} Let $R$ be a ring with unity and $U_r\subseteq R$ be the set of elements with right inverse. Then, the following statements are equivalent.

$(1)$ $R$ is a local ring.

$(2)$ $R{\setminus}U_r$ is contained in a right ideal of $R$.

$(3)$ $R{\setminus}U_r$ is a right ideal of $R$. 

$(4)$ If $R\neq0$, then $x$ or $1-x$ is a unit for every $x\in R$. 

\end{theorem}

The commutative analogue of the following theorem was proved in Lemma 2.26 of \cite{B}.
Here, we set a proof for duo rings.

\begin{theorem}\label{local1}
Let $R$ be a local ring with the maximal ideal $M$, such that every principal ideal of $R$ is almost prime. If $aR=Ra$ for all $a\in R$, then $M^2=0$.
\end{theorem}

\begin{proof} Let $a,b\in M{\setminus}0$. Then, the principle ideal $\langle ab\rangle$ is almost prime by assumption.
Assume that $\langle ab\rangle\neq0$. Then, $\langle a\rangle\langle b\rangle=\langle ab\rangle$, i.e., $\langle a\rangle\langle b\rangle\subseteq\langle ab\rangle$.

If $\langle a\rangle\langle b\rangle\not\subseteq\langle ab\rangle^2$ then either $\langle a\rangle\subseteq\langle ab\rangle$ or $\langle b\rangle\subseteq\langle ab\rangle$. 
This implies that $\langle a\rangle=\langle ab\rangle$ or $\langle b\rangle=\langle ab\rangle$. 
If $\langle a\rangle=\langle ab\rangle$, then $a=abr$ for some $r\in R$ due to the fact that $(ab)R=R(ab)$.
Hence, $a(1-br)=0$. Therefore, either $1-br=0$ or $1-br$ is a zero divisor. Clearly $1-br$ is not a unit, and so $br\in M$ is a unit by $(4)$ of Theorem \ref{local}.
Thus, $M=R$ which is a contradiction. Similarly, if $\langle b\rangle=\langle ab\rangle$, then we have the same contradiction.

So $\langle ab\rangle=\langle a\rangle\langle b\rangle\subseteq\langle ab\rangle^2=\langle (ab)^2\rangle$.
Then, $ab=(ab)^2s$ for some $s\in R$.
Thus, $ab(1-abs)=0$. If $ab\neq 0$ then similar to the first case we get again the contradiction $M=R$.
Hence, $ab=0$ which implies that $M^2=0$.
\end{proof}

\begin{remark} In the proof of Theorem \ref{local1}, one can conclude from $\langle a\rangle=\langle ab\rangle$,
$\langle b\rangle=\langle ab\rangle$ or $\langle ab\rangle= \langle ab\rangle^2$ that
\[
\langle a\rangle=0\ , \ \ \langle b\rangle=0 \ \ \text{\rm or} \ \ \langle ab\rangle=0,
\]
directly using Nakayama's lemma (Lemma 4.22 in \cite{L}).
\end{remark}

\subsection{Decomposable rings}

\begin{definition} A ring $R$ is called decomposable if $R=R_1\times R_2$ for some nontrivial rings $R_1$ and $R_2$.
\end{definition}

It is well known that any (right) ideal $R_1\times R_2$ has the form $I \times J$, where $I$ and $J$ are (right) ideals of $R_1$ and $R_2$, respectively.
In addition, the ring $R_1\times R_2$ is with unity (S-unital, regular) if and only if $R_1$ and $R_2$ are with unity (S-unital, regular). 

\begin{lemma}\label{1} Let $R_1$ and $R_2$ be noncommutative rings, and $P$ be a right ideal of $R_1$. Then the following statements are equivalent.   

$(1)$ $P$ is an almost prime right ideal of $R_1$.

$(2)$ $P\times R_2$ is an almost prime right ideal of  $R_1\times R_2$.  
\end{lemma}

\begin{proof}  $(1)\Rightarrow (2)$ Let
$(A_1\times B_1)(A_2\times B_2)\subseteq(P\times R_2)$, and 
\[
(A_1\times B_1)(A_2\times B_2)\not\subseteq(P\times R_2)^2,
\]
where $A_1, A_2$ are right ideals of $R_1$, and $B_1, B_2$ are right ideals of $R_2$. Then 
\[
(A_1A_2)\times(B_1B_2)\subseteq P\times R_2,
\]
and $(A_1A_2)\times (B_1B_2)\not\subseteq(P^2\times R_2^2)$.
Thus, $A_1A_2\subseteq P$ and $A_1A_2\not\subseteq P^2$. Hence, by $(1)$ either $A_1\subseteq P$ or $A_2\subseteq P$.
This implies that either $A_1\times B_1\subseteq P\times R_2$ or $A_2\times B_2\subseteq P\times R_2$.

$(2)\Rightarrow (1)$ Let $I,J$ be right ideals of $R_1$ such that $IJ\subseteq P$ and $IJ \not\subseteq P^2$. Then
\[
(I\times R_2)(J\times R_2)\subseteq(P\times R_2),
\]
and $(I\times R_2)(J\times R_2)\not\subseteq(P\times R_2)^2$.
Thus, by $(2)$ either $I\times R_2\subseteq P\times R_2$ or $J\times R_2 \subseteq P\times R_2$, which implies that either $I \subseteq P$ or $J \subseteq P$.
\end{proof}

\begin{remark}\label{r1} As a modification of Lemma \ref{1}, one can easily show that if $P$ is a right ideal of $R_2$, then $P$ is an almost prime right ideal of $R_2$ if and only if  $R_1\times P$ is an almost prime right ideal of $R_1\times R_2$.
\end{remark}

\begin{theorem}\label{2} Let $P$ be a right ideal of $R=R_1\times R_2$, where $R_1$ and $R_2$ are right \text{\rm S}-unital rings. Then $P$ is an almost prime right ideal of $R$ if and only if it has one of the following forms. 

$(1)$ $I\times R_2$, where $I$ is an almost prime right ideal of $R_1$.

$(2)$ $R_1\times J$, where $J$ is an almost prime right ideal of $R_2$.
 
$(3)$ $I\times J$, where $I$ and $J$ are idempotent right ideals of $R_1$ and $R_2$, respectively.
\end{theorem}

\begin{proof} $(\Rightarrow)$ Suppose that $P$ is an almost prime right ideal of $R$, then $P= I\times J$ for some right ideals $I$ and $J$ of $R_1$ and $R_2$, respectively.
If $I=R_1$ then $J$ is an almost prime right ideal of $R_2$ by Remark \ref{r1}. Similarly, if $J=R_2$ then $I$ is an almost prime right ideal of $R_1$ by Lemma \ref{1}.

Now suppose that $I\neq R_1$ and $J \neq R_2$. Without loss of generality, assume that $I$ is not an idempotent right ideal of $R_1$.
Then, there exists an element $x\in I{\setminus I^2}$. Thus, $(x\rangle\subseteq I$ and $(x\rangle\not\subseteq I^2$. Therefore,
\[
(x\rangle\times0\subseteq I\times J \ \ \text{\rm and} \ \  (x\rangle\times0\ \not\subseteq I^2 \times J^2.
\]
Now assume that $((x\rangle\times R_2)(R_1\times0)\subseteq I^2 \times J^2$.
Because $R_1$ is right S-unital, we get that $(x\rangle=(x\rangle R_1$.
Thus
\[
(x\rangle\times0=(x\rangle R_1\times R_2\cdot0=((x\rangle\times R_2)(R_1\times0)\subseteq I^2\times J^2,
\]
which is a contradiction. Therefore, $((x\rangle\times R_2)(R_1\times0)\not\subseteq I^2 \times J^2$.
On the other hand, $((x\rangle\times R_2)(R_1\times0)=(x\rangle\times0\subseteq I\times J$.
Because the $I\times J$ is an almost prime right ideal, we have that
either $(x\rangle\times R_2\subseteq I\times J$ or $R_1\times0\subseteq I\times J$.
Hence, either $R_2=J$ or $R_1 =I$, that yields a contadiction. Thus $I$ is idempotent. Similarly we can see that $J$ is an idempotent right ideal of $R_2$.

$(\Leftarrow)$ Suppose that $P$ has one of the three forms.
Then we come up with the proof by Lemma \ref{1}, Remark \ref{r1} and by the fact that every idempotent right ideal is almost prime. 
\end{proof}

The following proposition was proved for direct product of commutative rings in \cite{B}. Here, we set a similar proof for noncommutative rings.

\begin{proposition}\label{prop21}
Let $R_1$ and $R_2$ be right \text{\rm S}-unital rings. If every right ideal of $R_1$ and $R_2$ is a product of almost prime right ideals,
then every right ideal of $R_1\times R_2$ is a product of almost prime right ideals. 
\end{proposition}

\begin{proof}
Let $I$ and $J$ be right ideals of $R_1$ and $R_2$, respectively. Let
$I=A_1\cdots A_n$ and $J=B_1\cdots B_m$ for almost prime right ideals $A_i$ and $B_j$,
and let $P$ be a right ideal of $R_1\times R_2$.
Then $P$ must have one of the following three forms by Theorem \ref{2}.

\noindent $(1)$ If $P=I \times R_2$, then 
\[
P=(A_1\cdots A_n)\times R_2=(A_1\times R_2)\cdots(A_n\times R_2).
\]
\noindent $(2)$ If $P=R_1\times J$, then
\[
P=R_1\times(B_1\cdots B_m)=(R_1\times B_1)\cdots(R_1\times B_m).
\]
\noindent $(3)$ Finally, if $P=I\times J$, then
\begin{align}
P=&(A_1\cdots A_n)\times(B_1\cdots B_m)=(A_1\times R_2)\cdots(A_n\times R_2)(R_1\times B_1)\cdots(R_1\times B_m).\nonumber
\end{align}
In all cases we get a product of almost prime right ideals of $R$ due to Theorem \ref{2}.
\end{proof}

Note that every fully right  idempotent ring is a fully almost prime right ring. However, the converse does not hold in general. See \cite{AbouFindik} (Example 2.1) for a counterexample.
In the following theorem we show that the equivalence holds in direct product of S-unital rings.

\begin{theorem}\label{3}
Let $R=R_1\times R_2$, where $R_1$, $R_2$ are  right \text{\rm S}-unital rings. Then the following statements are equivalent. 

$(1)$ $R$ is a fully right idempotent ring.

$(2)$ $R$ is a fully almost prime right ring.
\end{theorem}

\begin{proof}
$(1)\Rightarrow (2)$ is trivial by the fact that every idempotent right ideal is an almost prime right ideal.

$(2)\Rightarrow (1)$ Assume that $P$ is a right ideal of $R$ which is not idempotent.
Then by Theorem \ref{2}, we have three cases for the form of $P$.

Case 1. $P=I\times R_2$, where $I$ is an almost prime right ideal of $R_1$.
Then $I$ is not idempotent ideal of $R_1$. So there exist an element $x\in I{\setminus}I^2$,
which yields that $(x\rangle\subseteq I$ and $(x\rangle\not\subseteq I^2$.
Hence,
\[
(x\rangle\times0\subseteq I\times0 \ \ \text{\rm and} \ \   (x\rangle\times0\not\subseteq I^2 \times0.
\]
Thus $((x\rangle\times R_2)(R_1\times0)\subseteq I\times0$ and
\[
((x\rangle\times R_2)(R_1\times0)\not\subseteq I^2 \times0^2=(I\times0)^2.
\]
Since $I\times0$ is an almost prime right ideal, then either $(x\rangle\times R_2\subseteq I\times0$ or
$R_1\times0\subseteq I\times0$. This implies that either $R_2=0$ or $I=R_1$ which is a contradiction. So $I$ must be idempotent.

Case 2. $P=R_1\times J$ where $J$ is an almost prime right ideal of $R_2$. Then similar to the proof of case $(1)$, one can show that $J$ must be idempotent.

Case 3. $P=I\times J$, where $I$ and $J$ are idempotent right ideals of $R_1$ and $R_2$, respectively. Then clearly $P$ is an right idempotent ideal of $R$.
\end{proof}

\begin{corollary}\label{corollarynew}
Let $R_1$ and $R_2$ be right \text{\rm S}-unital rings. Then the following statements are equivalent. 

$(1)$ $R_1\times R_2$ is a fully almost prime right ring.

$(2)$ $R_1$ and $R_2$ are fully right idempotent rings.
\end{corollary}

\begin{remark}\label{newest}
Note that all the results in Lemma \ref{1}, Remark \ref{r1}, Theorem \ref{2}, Proposition \ref{prop21}, Theorem \ref{3}, and Corollary \ref{corollarynew}
hold also for two sided ideals.
\end{remark}

\noindent Recall that an ideal $P$ of a ring $R$ is called semiprime, if $I^2\subseteq P$ implies $I\subseteq P$ for ideals $I$ of $R$.
The next result considers two sided ideals of decomposable rings.

\begin{corollary} Let $R=R_1\times R_2$, where $R_1$, $R_2$ are  right \text{\rm S}-unital rings. If the set of the ideals of R is totally ordered then the following statements are equivalent. 

$(1)$ $R$ is a fully almost prime ring.

$(2)$ $R$ is a fully idempotent ring.  

$(3)$ $R$ is a fully prime ring.  

$(4)$ $R$ is a fully semiprime ring.
\end{corollary}

\begin{proof}$(1)\Rightarrow (2)$ Straightforward by Remark \ref{newest}. 

$(2)\Rightarrow (1)$ Trivial by definition. 

$(2)\Rightarrow (3)$ A consequence of Theorem 1.2 of \cite {Bl}.

$(3)\Rightarrow (4)$ Trivial by definition. 

$(4)\Rightarrow (2)$ Let $P$ be an ideal of $R$. Because $P^2$ is semiprime, $P^2\subseteq P^2$ implies $P\subseteq P^2\subseteq P$, and thus $P=P^2$. 
\end{proof}

\begin{theorem}\label{4}
Let $R=\prod_{i=1}^n R_i$ for the rings $R_1,\ldots,R_n$. If R is a fully almost prime right ring, then so is $R_i$ for every $i=1,\ldots,n$.
\end{theorem}

\begin{proof}
Let $\pi_i:R\to R_i$ be the projective epimorphism, where
\[
\pi_i(a_1,a_2,\ldots,a_i,\ldots,a_n)=a_i,
\]
for all $i=1,\ldots,n$. Then, $R_i$ is a fully almost prime right ring for all $i=1,\ldots,n$, by Theorem 3.4 of \cite{AbouFindik}.
\end{proof}

\begin{remark}
Theorem \ref{2} shows that the converse of Theorem \ref{4} is not true in general,
in other words, the direct product of fully almost prime right rings need not to be a fully almost prime right ring.
The following corollary gives a special case such that the direct product of fully almost prime right rings is a fully almost prime right ring.
\end{remark}

\begin{corollary}\label{c411}
Let $R$ be a fully almost prime right ring, and let $I_1$ and $I_2$ be comaximal ideals of $R$.
Then
\[
R/I_1 \ , \ \ R/I_2\ , \ \ (R/I_1)\times(R/I_2)
\]
are fully almost prime right rings.
\end{corollary}

\begin{proof} $R/I_1$ and $R/I_2$ are fully almost prime right rings by Theorem 3.6 of \cite{A}.
Next, by the epimorphism $\phi:R\rightarrow (R/I_1)\times(R/I_2)$, defined as $\phi(r)=(r+I_1,r+I_2)$, the proof completes due to Theorem 3.4 of \cite{AbouFindik}.
\end{proof}

The following theorem is known as the Chinese remainder theorem.

\begin{theorem}
Let $R$ be a right \text{\rm S}-unital ring, and let $I_1,\ldots, I_n$ be pairwise comaximal ideals of $R$, such that $I_iI_j=I_jI_i$ for all $i,j\in\{1, 2, ..., n\}$.
Then,

$(1)$ $I_1\cdots I_n=I_1\cap\cdots\cap I_n$.

$(2)$ $R/(I_1\cdots I_n)= R/(I_1\cap\cdots\cap I_n)\cong(R/I_1)\times\cdots\times(R/I_n)$.
\end{theorem}

\begin{corollary}
Let $R$ be fully almost prime ring, which is right \text{\rm S}-unital.
If $M_1$ and $M_2$ are two distinct maximal ideals of $R$, then $R/(M_1M_2)=R/(M_1\cap M_2)$ is a fully almost prime ring.
\end{corollary}

\begin{proof}
Firstly, $M_1$ and $M_2$ are comaximal ideals and thus by Corollary \ref{c411},
we obtain that $(R/M_1)\times(R/M_2)$ is a fully almost prime ring.
Now, by Theorem \ref{fully two sided} we have $M_1M_2=M_2M_1$. Thus by Chinese remainder theorem we are done.  
\end{proof}

Theorem 1.3 of \cite{Bl} states that the center $C(R)$ of a fully prime (noncommutative) ring $R$ is either a field or the zero ring.
This holds also for fully prime right rings.
However, the ring of polynomials over a fully prime right ring is not a fully prime right ring, since the polynomial $x\in C(R[x])$ that is not a unit.

In the following theorem, we show that if the direct product $R=R_1\times R_2$ of right S-unital rings $R_1$ and $R_2$ is a fully almost prime right ring
such that $0$ is a prime right ideal, then the ring $R[x]$ of polynomials fails to be a fully almost prime right ring.
The proof is very similar to the proof of Theorem 1.3 of \cite{Bl}.

\begin{theorem}
Let $R=R_1\times R_2$ be a fully almost prime right ring, where $R_1$ and $R_2$ are right \text{\rm S}-unital rings.
If $0$ is a prime right ideal of $R$ and $C(R)\neq0$, then $R$ is with unity and $C(R)$ is a field. 
\end{theorem}

\begin{proof} By Theorem \ref{3}, $R$ is fully right idempotent ring.
 Let $0\neq a\in C(R)$.
Then,
\[
a^2R=aRa\subseteq aR=(aR)^2=aR(aR)=a^2RR\subseteq a^2R. 
\]
Thus, $a^2R=aR$. Hence $a\in a^2R$, which means that $a=a^2b$ for some $b\in R$.
Let $ab=e$ for some $e\in R$. Then, $ae=a(ab)=a^2b=a$, and
\[
a(er-re)=a(er)-a(re)=(ae)r-r(ae)=ar-ra=0,
\]
for all $r\in R$. Hence $aR(er-re)=Ra(er-re)=0$. Therefore, $er=re$ since 0 is a prime right ideal. This yields that $e\in C(R)$.

Similarly, $a(r-re)=0$ and this implies that $aR(r-re)=Ra(r-re)=0$. Thus, $r=re=er$ for all $r\in R$.
Hence, $e=ab=1\in C(R)$. 

Moreover, $a(br-rb)=abr-arb=abr-rab=0$. Thus $aR(br-rb)=Ra(br-rb)=0$, which implies that $br=rb$ for all $r\in R$. Finally, $b\in C(R)$ and hence $a$ is invertible in $C(R)$.
This implies that $C(R)$ is a field. 
\end{proof}

\begin{remark}
Consider $R=R_1\times R_2$ as a fully almost prime right ring, where $R_1$, $R_2$ are right $S$-unital rings. Then $R_1$, $R_2$ are both  fully almost prime right ring, by Corollary \ref{corollarynew}. Consider the ring $R_1 [x]\times R_2 [x]$, denoted as $R[x]$, which is formed as a direct product of two right $S$-unital polynomials rings. In the case that $R[x]$ is a fully almost prime right ring with 0 being a prime right ideal of $R[x]$, it follows that $C(R[x])$ is a field, however, $(x,x)\in C(R)$, contradiction, since $(x,x)$ has no inverse. Hence, $R[x]$ does not meet the criteria to be classified as a fully almost prime right ring
\end{remark}
\section{Almost right $S$-prime ideals}

We first menition the definition of right $S$-prime ideal.

\begin{definition}[Definition 2.2 of \protect\cite{A}]
Let $P$ be an ideal of a ring $R$, and $S$ be an $m$-system of $R$ such that 
$P\cap S=\phi$. We call $P$ a right $S$-prime ideal associated with an
element $s\in S$ (briefly a right $S$-prime ideal), if for some ideals $A$, $%
B$ of $R$ with $AB \subseteq P$, either $A\langle s\rangle \subseteq P$ or $%
B\langle s\rangle\subseteq P$.
\end{definition}

\begin{definition}
Let $R$ be a ring and $S$ be an $m$-system of $R$. An ideal $P$ of $R$ with $%
P\cap S=\phi $ is called a weakly right $S$-prime ideal associated with an
element $s\in S$ (or briefly a weakly right $S$-prime ideal) if whenever $A$
and $B$ are ideals of $R$ with $0\neq AB\subseteq P$, implies either $%
A\langle s\rangle \subseteq P$ or $B\langle s\rangle \subseteq P$.
\end{definition}

\begin{definition}
Let $R$ be a ring and $S$ be an $m$-system of $R$. An ideal $P$ of $R$ with $%
P\cap S=\phi $ is called an almost right $S$-prime ideal associated with an
element $s\in S$ (or briefly an almost right $S$-prime ideal) if whenever $A$
and $B$ are ideals of $R$ with $AB\subseteq P$ and $AB\nsubseteq P^{2}$,
implies either $A\langle s\rangle \subseteq P$ or $B\langle s\rangle
\subseteq P$.
\end{definition}

\begin{remark}
\label{11} Let $S$ denote an $m$-system of a ring $R$, and suppose $P$ is an almost right $S$-prime ideal of $R$. Let $\mathfrak{J}$ represent the collection of ideals that do not intersect with $S$, then $P$ belongs to $\mathfrak{J}$ and $\mathfrak{J}$ is not empty. By applying Zorn's lemma, we infer that $\mathfrak{J}$ possesses a maximal element $J$. Consequently, by employing Proposition 10.5 from \cite{L}, we obtain that $J$ is prime. Therefore, an almost right $S$-prime ideal $P$ of a ring $R$, which is maximal concerning the condition that $P$ does not intersect with $S$, is indeed a prime ideal, also known as an almost prime ideal. Additionally, if $R$ possesses an identity element and $S$ is a subset of $U(R)$ (the set of units in $R$), then the concepts of almost prime ideals and almost right $S$-prime ideals coincide. Furthermore, if $S_{1}$ and $S_{2}$ are two $m$-systems of $R$ with $S_{1}$ being a subset of $S_{2}$, then every almost right $S_{1}$-prime ideal (disjoint from $S_{2}$) also qualifies as an almost right $S_{2}$-prime ideal.
\end{remark}

\begin{example}
$(1)$ From the definition, it is clear that every (almost) prime ideal that does not intersect with $S$ is also an almost right $S$-prime ideal. Additionally, every right $S$-prime ideal (as defined in \cite{A}) is also an almost right $S$-prime ideal. However, the zero ideal presents an example of a weakly right $S$-prime ideal, and consequently an almost right $S$-prime ideal, which might not always be a right $S$-prime ideal in all cases.

$(2)$ Let $R$ be a local ring with the unique maximal ideal $M$. Then, $R{%
\setminus }M$ is an $m$-system that is disjoint from each ideal of $R$.
Thus, $M$ is an almost right $(R{\setminus }M)$-prime ideal. In addition, by
Remark \ref{11}, the almost right $(R{\setminus }M)$-prime ideals are
precisely the almost prime ideals of $R$.

$(3)$ Let $R=M_{2}(\mathbb{Z}_{12})$. The product of the non zero ideals $%
A=M_{2}(\langle 2\rangle )$ and $B=M_{2}(\langle 6\rangle )$ equals $0$,
hence $0$ is not a prime ideal. In addition, let 
\begin{equation*}
S=\Bigg\{s=\left( 
\begin{array}{cc}
r & 0 \\ 
0 & r%
\end{array}%
\right) ;r\in \{3,9\}\Bigg\},
\end{equation*}%
then $S$ is an $m$-system of the ring $R$. The zero ideal, of $R$, is not
right $S$-prime becuase $A\langle s\rangle \neq 0$ and $B\langle s\rangle
\neq 0$ for all $s\in S$. However, clearly $0$ is an almost right $S$-prime
ideal.

$(4)$ Let $R=M_{2}(\mathbb{Z}[x])$, and $P=M_{2}(\langle 9x\rangle )$.
Example 2.12 of \cite{A}, showed that the ideal $P$ is a right $S$-prime
ideal (i.e., an almost right $S$-prime ideal), but not prime, considering 
\begin{equation*}
S=\Bigg\{s^{2^{n}},\text{ where }s=\left( 
\begin{array}{cc}
3 & 0 \\ 
0 & 3%
\end{array}%
\right) ,n\in \mathbb{Z}^{+}\Bigg\}.
\end{equation*}%
In addition, $P$ is not an almost prime ideal, because for the ideals $%
A=M_{2}(\langle 9\rangle )$ and $B=M_{2}(\langle x\rangle )$, we have $%
AB\subseteq P$ and $AB\not\subseteq P^{2}$, however, $A\not\subseteq P$ and $%
B\not\subseteq P$.
\end{example}

\begin{theorem}
\label{BE} Let $S\subseteq R$ be an $m$-system of a ring $R$ with identity,
and $P$ be an ideal of $R$ such that $P\cap S=\phi$, then the following
statements are equivalent:

$(1)$ $P$ is a almost right $S$-prime ideal.

$(2)$ For $a,b\in R$ with $aRb\subseteq P$ and $aRb\nsubseteq P^{2}$, we
have $a\langle s\rangle \subseteq P$ or $b\langle s\rangle \subseteq P$.
\end{theorem}

\begin{proof}
$(1)\Rightarrow (2)$: Let $aRb\subseteq P$ and $aRb\nsubseteq P^{2}$ for
some $a,b\in R$. Then we have $\langle a\rangle \langle b\rangle
=RaRRbR\subseteq P$ and $\langle a\rangle \langle b\rangle \nsubseteq P^{2}$
since $aRb\nsubseteq P^{2}$. By assumption $(1)$, either $a\langle s\rangle
\subseteq \langle a\rangle \langle s\rangle \subseteq P$ or $b\langle
s\rangle \subseteq \langle b\rangle \langle s\rangle \subseteq P$, for some $%
s\in S$.

$(2)\Rightarrow (1)$: Consider ideals $A$ and $B$ of $R$ with $AB\subseteq P$
and $AB\nsubseteq P^{2}$. Suppose that $A\langle s\rangle \not\subseteq P$
and $B\langle s\rangle \not\subseteq P$ for all $s\in S$. In this case, we
can choose $a\in A-P$ and $b\in B-P$. For some $a_{1}\in P\cap A$ and $%
b_{1}\in P\cap B$, we have $(a+a_{1})R(b+b_{1})\subseteq P$, and since $%
(a+a_{1})\langle s\rangle \not\subseteq P$ and $(b+b_{1})\langle s\rangle
\not\subseteq P$ for any $s\in S$, it implies $(a+a_{1})R(b+b_{1})\subseteq
P^{2}$. By considering all possible cases where $a_{1}$ and or $b_{1}$ are
elements of $P^{2},$ $ab_{1}\in P^{2},$ $a_{1}b\in P^{2},$ $a_{1}b_{1}\in
P^{2}$. Hence $ab\in P^{2}.$ Thus, $AB\subseteq P^{2}$, which leads to a
contradiction.
\end{proof}

 Let $S$ be an $m$-system of a ring $R$.
In referring to \cite{A}, Definition 3.1 states that   an  ideal  $I$ of $R$ is called an $S$-finite  ideal   if there is a finitely generated  ideal $J$ such that $I\langle s\rangle\subseteq J\subseteq I$ for some $s\in S$. On the other hand, every idempotent ideal is clearly  almost prime, and if it is disjoint from $S$, then, it is an almost right $S$-prime ideal. 

\begin{corollary} Let $S$ be an $m$-system of a ring $R$, and let $A, B$ be non $S$-finite ideals of $R$ such that $A\cap B=0$. If $AB$ is an almost right $S$-prime ideal, then, it is an idempotent ideal.  
\end{corollary}
\begin{proof} If $A\langle s\rangle\subseteq AB$ for some $s\in S$, then,  $A\langle s\rangle\subseteq A\cap B\subseteq A$, contradiction since $A$ is not $S$-finite. Hence,  $A\langle s\rangle\not\subseteq AB$ for all $s\in S$. Similarly, one can see $B\langle s\rangle\not\subseteq AB$ for all $s\in S$. Since $AB=AB$, and $AB$ is an almost right $S$-prime, then, $AB\subseteq (AB)^2 \subseteq AB$. Thus, $AB$ is idempotent. 
\end{proof}

\begin{proposition}\label{finite} Let $S$ be an $m$-system of a local ring $(R, M)$, and let $P$ be  an  ideal of $R$, disjoint from $S$, such that $P^2=M^2$. Then, $P$ is an almost right $S$-prime ideal. 
\end{proposition}

\begin{proof}
The proof is consequences of Proposition \ref{local2}.
\end{proof}

\begin{proposition}
Let $S$ be an $m$-system of a local ring $(R, M)$, and let $P$ be a finitely generated   ideal of $R$, disjoint from $S$, such that $M^2\subseteq P$ and $M^2\neq P^2$. If $P$ is an almost right $S$-prime  ideal, then, $M$ is $S$-finite. 
\end{proposition}

\begin{proof}
Suppose that $P$ is an almost right $S$-prime  ideal of $R$.  Clearly, $P^2\subseteq M^2$. 
If $M^2\subseteq P^2$, then,  $M^2 =P^2$, contradiction. Thus,   $M^2\not\subseteq P^2$, and $M^2\subseteq P$, hence by assumption,  $M \langle s\rangle\subseteq P\subseteq M$ for some $s\in S$. 
\end{proof}

\begin{theorem}
\label{ID} Let $S\subseteq R$ be an $m$-system of a ring $R$ with identity,
and $P$ be an ideal of $R$ such that $P\cap S=\phi$, then the following
statements are equivalent:

$(1)$ $P$ is an almost right $S$-prime ideal associated with $s\in S$.

$(2)$ For all $r\in R-(P:\langle s\rangle )$ either $(P:\langle r\rangle
)\subseteq (P:\langle s\rangle )$ or $(P:\langle r\rangle )=(P^{2}:\langle
r\rangle )$.
\end{theorem}

\begin{proof}
$(1)\Rightarrow(2)$: Let $y \in (P:\langle r\rangle)$, then $y\langle
r\rangle\subseteq P$. Here we discuss two cases.

If $y\langle r\rangle \nsubseteq P^{2}$, then $\langle y\rangle \langle
r\rangle \subseteq P$ and $\left\langle y\right\rangle \langle r\rangle
\nsubseteq P^{2}$, and hence by $(1)$ either $\langle y\rangle \langle
s\rangle \subseteq P$ or $\langle r\rangle \langle s\rangle \subseteq P$.
Consequently, either $r\in \langle r\rangle \subseteq (P:\langle s\rangle )$%
, contadiction, or $y\in \langle y\rangle \subseteq (P:\langle s\rangle )$.

If $y\langle r\rangle \subseteq P^{2}$, then $y\in (P^{2}:\langle r\rangle )$%
.

Thus, $(P:\langle r\rangle )\subseteq (P:\langle s\rangle )\cup
(P^{2}:\langle r\rangle )$. Hence, either $(P:\langle r\rangle )\subseteq
(P:\langle s\rangle )$ or $(P:\langle r\rangle )=(P^{2}:\langle r\rangle )$,
(recall that $(P^{2}:\langle r\rangle )\subseteq (P:\langle r\rangle )$).

$(2)\Rightarrow (1)$: Suppose that $aRb\subseteq P$ and $aRb\nsubseteq P^{2}$
for some $a,b\in R$. If $b\langle s\rangle \not\subseteq P$, then $b\not\in
(P:\langle s\rangle )$, hence, by $(2)$, either $(P:\langle b\rangle
)\subseteq (P:\langle s\rangle )$ or $(P:\langle b\rangle )=(P^{2}:\langle
b\rangle )$. On the other hand, we have $a\in (P:\langle b\rangle )$, thus,
either $a\in (P^{2}:\langle b\rangle )$, hence $aRb\subseteq a\langle
b\rangle \subseteq P^{2}$, contradiction. Or $a\in (P:\langle s\rangle )$,
hence, $a\langle s\rangle \subseteq P$. Consequently, $P$ is an almost right 
$S$-prime ideal.
\end{proof}

\begin{corollary}
\label{NICO}[Theorem 1.14 of \cite{AbouFindik} For an ideal $P$ of a ring $R$
with identity, the following statements are equivalent:

$(1)$ $P$ is almost prime.

$(2)$ For $r\in R-P$, $(P:\left\langle r\right\rangle )=P\cup
(P^{2}:\left\langle r\right\rangle )$.

$(3)$ For $r\in R-P$, $(P:\left\langle r\right\rangle )=P$ or $%
(P:\left\langle r\right\rangle )=(P^{2}:\left\langle r\right\rangle )$.
\end{corollary}

\begin{proof}
Follows from Theorem \ref{ID} by taking $S=\{1\}$. Recall that $(P:R)=P$ and 
$P\subseteq (P:\langle r\rangle )$ for all $r\in R$.
\end{proof}

Let $S\subseteq R$ be an $m$-system of a ring $R$ with identity, and $P$ be
an ideal of $R$ such that $P\cap S=\phi $, then an almost right $S-$prime
ideal does not have to be a right $S-$prime ideal or a weakly right $S$%
-prime ideal. In the next two theorems, we examine some cases such that the
statements above are equivalent.

\begin{theorem}
Let $S$ be an $m$-system of a unitary ring $R$, and $P$ be
an ideal of $R$ such that $P\cap S=\phi $ and such that $(P^{2}:P)\subseteq
P $. Then the following statements are equivalent.
\end{theorem}

(1) $P$ is an almost right $S$-prime ideal.

(2) $P$ is a right $S$-prime ideal.

\begin{proof}

$(1)\Rightarrow (2)$ Suppose the almost right $S$-prime ideal $P$ is not a right $S$-prime ideal. This implies the existence of some ideals $A$ and $B$ of $R$ such that $AB\subseteq P$ with $A\left\langle s\right\rangle \nsubseteq P$ and $B\left\langle s\right\rangle \nsubseteq P$ for all $s\in S$. Consequently, from  $(1)$, we deduce that $AB\subseteq P^{2}$. Now, let $a\in A\backslash P$ and $b\in B\backslash P$ such that $a\left\langle s\right\rangle \nsubseteq P$ and $b\left\langle s\right\rangle \nsubseteq P$ for all $s\in S$. Then, we have
\[
(\left\langle
a\right\rangle +P)\left\langle b\right\rangle =\left\langle a\right\rangle
\left\langle b\right\rangle +P\left\langle b\right\rangle \subseteq
AB+P\subseteq P.
\]
If $(\left\langle a\right\rangle +P)\left\langle
b\right\rangle \subseteq P^{2}$, then $P\left\langle b\right\rangle
\subseteq P^{2}$. This implies that $b\in (P^{2}:P)\subseteq P$. This
contradicts with $b\in B\backslash P$. If $(\left\langle a\right\rangle
+P)\left\langle b\right\rangle \nsubseteq P^{2}$, then by (2) we get that
either $(\left\langle a\right\rangle +P)\left\langle s\right\rangle
\subseteq P$ or $\left\langle b\right\rangle \left\langle s\right\rangle $ $%
\subseteq P$, which implies $a\left\langle s\right\rangle \subseteq P$ or $%
a\left\langle s\right\rangle \subseteq P$ respectively. Contradiction.

$(2)\Rightarrow (1)$ Clear by definition.
\end{proof}

\begin{theorem}
Let $S\subseteq R$ be an $m$-system of a ring $R$, and $P$ be
an ideal of $R$ such that $P\cap S=\phi $, and such that $P^{2}=0.$ Then, the
following statements are equivalent.
\end{theorem}

(1) $P$ is a weakly right $S$-prime ideal.

(2) $P$ is an almost right $S$-prime ideal.

\begin{proof}
$(1)\Rightarrow (2)$ Follows from definition.

$(2)\Rightarrow (1)$ Suppose that $A$ and $B$ are ideals of $R$ such that $%
0\neq AB\subseteq P$. Then $AB\nsubseteq P^{2}=0$. Thus, by (2) we are done.
\end{proof}

The next result is a consequence of the Theorem above.

\begin{corollary}
Let $S\subseteq R$ be an $m$-system of a ring $R$, in which $R^{2}=0$, and  $P$ be
an ideal of $R$ such that $P\cap S=\phi $. Then, the
following statements are equivalent.
\end{corollary}

(1) $P$ is a weakly right $S$-prime ideal.

(2) $P$ is an almost right $S$-prime ideal.

\begin{lemma}
\label{AA} Let $S$ be an $m$-system of a ring $R$. Then we have the
following.

$(1)$ For some rings epimorphism $f$: $R\to f(R)$, f(S) is an $m$-system of $%
f(R)$.

$(2)$ The set $\bar{S}=\{\bar{s}=s+I:s\in S\}$, is an $m$-system of $R/I$
for any ideal $I$ of $R$.
\end{lemma}

The following theorem, is the $S$-version of Theorem 2.11 of \cite{AbouFindik}. 
\begin{theorem}
Let $S\subseteq R$ be an $m$-system of a ring $R$ with identity, and $P$ be
an ideal of $R$ such that $P\cap S=\phi $, then the following statements are
equivalent
\end{theorem}

(1) $P$ is an almost right $S$-prime right ideal of $R$.

(2) $P/P^{2}$ is a weakly right $S$-prime ideal of $R/P^{2}$.

\begin{proof}
$(1)\Rightarrow (2)$ Suppose that $\overline{0}\neq \left( a+P^{2}\right)
R/P^{2}\left( b+P^{2}\right) \subseteq P/P^{2}$. Thus, $aRb\subseteq P$ and $aRb\nsubseteq P^{2}$. Now by (1) we have that either $%
a\left\langle s\right\rangle \subseteq P$ or $b\left\langle s\right\rangle
\subseteq P$ for some $s\in S$. Hence, $\left( a+P^{2}\right) [(\langle
s\rangle +P^{2})/P^2]\subseteq P/P^{2}$ or $\left( b+P^{2}\right) [(\langle
s\rangle +P^{2})/P^2]\subseteq P/P^{2}.$ Hence $P/P^{2}$ is a weakly right $S$%
-prime ideal of $R/P^{2}$.

$(2)\Rightarrow (1)$ Suppose that $P/P^{2}$ is a weakly right $S$-prime ideal in $%
R/P^{2}$ and let $aRb\subseteq P$ and $aRb\nsubseteq P^{2},$ then $0$ $\neq $
$(a+P^{2})R/P^{2}(b+P^{2})\subseteq P/P^{2}$. But $P/P^{2}$ is a weakly
right $S$-prime ideal in $R/P^{2},$ then,
\begin{equation*}
 \text{either }
 a\left\langle s\right\rangle
+aP^{2}=(a+P^{2})[(\langle s\rangle +P^{2})/P^2]\subseteq P/P^{2} 
\end{equation*}

\begin{equation*}
\text{ or } b\left\langle
s\right\rangle +bP^{2}=(b+P^{2})[(\langle s\rangle +P^{2})/P^2]\subseteq P/P^{2}. 
\end{equation*}
Thus, $%
a\left\langle s\right\rangle \subseteq $ $P$ or $b\left\langle
s\right\rangle \subseteq $ $P$ which means $P$ is an almost right $S$-prime
right ideal of $R.$
\end{proof}

\begin{theorem}
\label{fully1 copy(1)} Let $f:R_{1}\rightarrow R_{2}$ be a ring epimorphism.
Consider $S\subseteq R_{1}$ as an $m$-system, and let $P$ be an ideal of $%
R_{1}$ such that $\text{ker}(f)\subseteq P$. If $P$ is an almost right $S$%
-prime ideal of $R_{1}$ associated with $s\in S$, then $f(P)$ is an almost
right $f(S)$-prime ideal of $R_{2}$ associated with $f(s)$.
\end{theorem}

\begin{proof}
Assume that $P$ is an almost right $S$-prime ideal, and consider $%
B_{1}B_{2}\subseteq f(P)$ and $B_{1}B_{2}\nsubseteq f(P)^{2}$ for some
ideals $B_{1}$ and $B_{2}$ of $R_{2}$. Then, there exist ideals $A_{1}$ and $%
A_{2}$ of $R_{1}$ such that $A_{1}=f^{-1}(B_{1})$ and $A_{2}=f^{-1}(B_{2})$.
Since $f$ is an epimorphism, it follows that $f(A_{1})=B_{1}$ and $%
f(A_{2})=B_{2}$, implying that $f(A_{1}A_{2})=B_{1}B_{2}\subseteq f(P)$ and $%
f(A_{1}A_{2})=B_{1}B_{2}\nsubseteq f(P)^{2}$. Consequintly $%
A_{1}A_{2}\subseteq f^{-1}(f(A_{1}A_{2}))\subseteq f^{-1}(f(P))=P$ and $%
A_{1}A_{2}\nsubseteq P^{2}.$ Then, we have that either $A_{1}\langle s\rangle
\subseteq P$ or $A_{2}\langle s\rangle \subseteq P$ for some $s\in S$,
leading to either $B_{1}\langle f(s)\rangle \subseteq P$ or $B_{2}\langle
f(s)\rangle \subseteq P$. According to $(1)$ of Lemma \ref{AA}, $f(S)$ forms
an $m$-system of $R_{2}$. Suppose $f(S)\cap f(P)\neq \phi $, and let $y\in
f(S)\cap f(P)$. This implies the existence of $s_{1}\in S$ and $x\in P$ such
that $y=f(s_{1})=f(x)$. Consequently, $x-s_{1}\in \text{\textrm{ker}}%
f\subseteq P$, leading to $s_{1}\in P$. However, this contradicts the fact
that $P\cap S=\phi $. Therefore, $f(S)$ is disjoint from $f(P)$, and
consequently, $f(P)$ is an almost right $f(S)$-prime ideal of $R_{2}$.
\end{proof}

\begin{theorem}
\label{fully1 copy(2)} Let $f:R_{1}\rightarrow R_{2}$ be a ring
monomorphism. Consider $S\subseteq R_{1}$ as an $m$-system, and let $P$ be
an ideal of $R_{1}$. If $f(P)$ is an almost right $f(S)$-prime ideal of $%
R_{2}$ associated with $f(s)\in f(S)$, then $P$ is an almost right $S$-prime
ideal of $R_{1}$ associated with $s$.
\end{theorem}

\begin{proof}
Assume that $f(P)$ is an almost right $f(S)$-prime ideal of $R_{2}$, and let 
$A_{1}$, $A_{2}$ be ideals of $R_{1}$ with $A_{1}A_{2}\subseteq P$ and $%
A_{1}A_{2}\nsubseteq P^{2}$. Then, $f(A_{1})f(A_{2})=f(A_{1}A_{2})\subseteq
f(P)$. If $f(A_{1}A_{2})\subseteq f(P^{2})$, then $A_{1}A_{2}\subseteq
f^{-1}(f(A_{1}A_{2}))\subseteq f^{-1}(f(P^{2}))=P^{2}$, which is a
contradiction. Hence, $f(A_{1})f(A_{2})=f(A_{1}A_{2})\nsubseteq
f(P^{2})=f(P)^{2}$ . Since $f(P)$ an almost right $S$-prime ideal of $R_{1}$

thus either $f(A_{1})\langle m\rangle \subseteq f(P)$ or $f(A_{2})\langle
m\rangle \subseteq f(P)$ for some $m\in f(S)$. Hence either $A_{1}\langle
s\rangle \subseteq f^{-1}(f(A_{1}\langle s\rangle ))\subseteq f^{-1}(f(P))=P$
or $A_{2}\langle s\rangle \subseteq f^{-1}(f(A_{2}\langle s\rangle
))\subseteq f^{-1}(f(P))=P$, where $m=f(s)$ for some $s\in S$. If $P\cap
S\neq \phi $, then for some $x\in P\cap S$ we have $f(x)\in f(P\cap
S)\subseteq f(P)\cap f(S)=\phi $, contradiction. Thus, $P\cap S=\phi $ and
hence, $P$ is an almost right $S$-prime ideal of $R_{1}$.
\end{proof}

\begin{theorem}
\label{fully2} Let $I$ be an ideal of a ring $R$, and $S\subseteq R$ be an $%
m $-system. If $P$ is an ideal of $R$ such that $I\subseteq P$. Then

$(1)$ If $P$ is an almost right $S$-prime ideal of $R$, then $P/I$ is an
almost right $\bar{S}$-prime ideal of $R/I$.

$(2)$ If $P/I$ is an almost right $\bar{S}$-prime ideal of $R/I$, then $P$ is an almost right $S$%
-prime ideal of $R$.
\end{theorem}

\begin{proof}
$(1)$ Let $P$ be an almost right $S$-prime ideal of $R$ and let $\bar{A}\bar{%
B}\subseteq \overline{P}=P/I$ and $\bar{A}\bar{B}\nsubseteq \overline{P}%
^{2}=(P/I)^{2}$ for some ideals $\bar{A},\bar{B}$ of $R/I$. Then there exist
ideals $A$ and $B$ of $R$ containing $I$ such that $\bar{A}=A/I$ and $\bar{B}%
=B/I$. So, $(AB+I)/I\subseteq P/I$, and $(AB+I)/I\nsubseteq (P^{2}+I)/I,$
thus $AB\subseteq P$ and $AB\nsubseteq P^{2}$ so by assumption, there exists $%
s\in S$ such that either $A\langle s\rangle \subseteq P$ or $B\langle
s\rangle \subseteq P$, and hence either $[(A+I)/I][(\langle s\rangle
+I)/I]\subseteq P/I$ or $[(B+I)/I][(\langle s\rangle +I)/I]\subseteq P/I$.
Thus either $A\langle s\rangle \subseteq P$ or $B\langle s\rangle \subseteq
P $. By $(2)$ of Lemma \ref{AA}, $\bar{S}$ is an $m$-system of $R/I$ and $%
P/I\cap \bar{S}=\phi $. Hence, $P/I$ is an almost right $\bar{S}$-prime ideal
of $R/I$.

$(2)$ Let $P/I$ be an almost right $\bar{S}$-prime ideal, and suppose that $%
AB\subseteq P$ and $AB\nsubseteq P^{2}$ for some ideals $A$, $B$ of $R$,
then $[(A+I)/I][(B+I)/I]\subseteq P/I$ and $[(A+I)/I][(B+I)/I]\nsubseteq
(P/I)^{2}$ $=\overline{P}^{2}.$ Thus, either $[(A+I)/I]\langle \bar{s}\rangle
\subseteq P/I$ or $[(B+I)/I]\langle \bar{s}\rangle \subseteq P/I$, for some $%
s+I=\bar{s}\in \bar{S}$, then either $[(A+I)/I][(\langle s\rangle
+I)/I]\subseteq P/I$ or $[(B+I)/I][(\langle s\rangle +I)/I]\subseteq P/I$.
Hence, either $A\langle s\rangle \subseteq P$ or $B\langle s\rangle \subseteq
P$. By notice that $S\cap P=\phi $ and $S$ is an $m$-system, we obtain that $%
P$ is an almost right $S$-prime ideal of $R$.
\end{proof}

Let $R_{_{^{i}}}$ be rings with identities and $S_{i}\subseteq $ $%
R_{_{^{i}}} $be an $m$-system of $R_{i}$ for each $i=1,2,...,n$, where $n\in 
\mathbb{N}$. Suppose that $R=R_{1}\times R_{2}\times \cdots \times R_{n}$ and $S=S_{1}\times
S_{2}\times \cdots \times S_{n}$. Then, note that $S$ is an $m$-system of $R$.
Also, each ideal of $R$ has the form $P=P_{1}\times P_{2}\times \cdots
\times P_{n}$, where $P_{i}$ is an ideal of $R_{i}$ for each $i=1,2,...,n.$

\begin{theorem}
\label{product} Suppose $R_{i}$ are rings with identies and $P_{i}$ is an
ideal of $R_{i}.$ Suppose that $S_{i}$ $\subseteq R_{i}$ is an m-system of $%
R_{i}$ for each $i=1,2$. Let $R=R_{1}\times R_{2}$ and $S=S_{1}\times S_{2}$%
. Then following statements are equivalent

(1) $P$ is an almost right $S_{1}$-prime ideal of $R_{1}$.

(2) $P\times R_{2}$ is an almost right $S$-prime ideal of $R_{1}\times R_{2}$%
.
\end{theorem}

\begin{proof}
$(i)\Rightarrow (ii)$ Let $(A_{1}\times A_{2})(B_{1}\times B_{2})\subseteq
P\times R_{2}$ and $(A_{1}\times A_{2})(B_{1}\times B_{2})\nsubseteq
(P\times R_{2})^{2}=P^{2}\times R_{2},$ where $A_{1}$,$B_{1}$ are ideals
of $R_{1}$, and $A_{2}$, $B_{2}$ are ideals of $R_{2.}$ Then $%
A_{1}B_{1}\times A_{2}B_{2}\subseteq P\times R_{2}$ and $A_{1}B_{1}\times
A_{2}B_{2}\nsubseteq P^{2}\times R_{2}.$ Thus $A_{1}B_{1}\subseteq P$ and $%
A_{1}B_{1}\nsubseteq P^{2}.$ Since $P$ is an almost right $S_{1}$-prime
ideal of $R_{1},$ $A_{1}\left\langle s\right\rangle \subseteq P$ or $%
B_{1}\left\langle s\right\rangle \subseteq P$ for some $s\in S_{1}.$ Hence $%
A_{1}\left\langle s\right\rangle \times A_{2}\subseteq P\times R_{2}$ and $%
B_{1}\left\langle s\right\rangle \times B_{2}\subseteq P\times R_{2}.$ Thus $%
(A_{1}\times A_{2})(\left\langle s\right\rangle \times R_{2})\subseteq
P\times R_{2}$ or $(B_{1}\times B_{2})(\left\langle s\right\rangle \times
R_{2})\subseteq P\times R_{2}$. Now we have for any $s_{2}\in S_{2}$ that $%
(A_{1}\times A_{2})\left\langle (s,s_{2})\right\rangle =(A_{1}\left\langle
s\right\rangle \times A_{2}\left\langle s_{2}\right\rangle )\subseteq
P\times R_{2}$ or $(B_{1}\times B_{2})\left\langle (s,s_{2})\right\rangle
=(B_{1}\left\langle s\right\rangle \times B_{2}\left\langle
s_{2}\right\rangle )\subseteq P\times R_{2}.$ Since $P\cap S_{1}=\emptyset $
we have $P\times R_{2}\cap S=\emptyset .$ Hence $P\times R_{2}$ is an almost
right $S$-prime ideal of $R_{1}\times R_{2}$.

$(ii)\Rightarrow (i)$ Let $A_{1}$, $B_{1}$ be ideals of $R_{1}$, and $A_{2}$, $B_{2}$ be ideals of $R_{2,}$ such $A_{1}B_{1}\subseteq P$ and $%
A_{1}B_{1}\nsubseteq P^{2}.$ Then $(A_{1}\times R_{2})(B_{1}\times
R_{2})\subseteq P\times R_{2}$ and $(A_{1}\times R_{2})(B_{1}\times
R_{2})\nsubseteq (P\times R_{2})^{2}=P^{2}\times R_{2}.$ Thus $(A_{1}\times
R_{2})\left\langle (s_{1},s_{2})\right\rangle \subseteq P\times R_{2}$ or $%
(B_{1}\times R_{2})\left\langle (s_{1},s_{2})\right\rangle \subseteq P\times
R_{2}$ for some $\left( s_{1},s_{2}\right) \in S_{1}\times S_{2}.$ Then $%
A_{1}\left\langle s_{1}\right\rangle \times R_{2}\left\langle
s_{2}\right\rangle \subseteq P\times R_{2}$ or $B_{1}\left\langle
s_{1}\right\rangle \times R_{2}\left\langle s_{2}\right\rangle \subseteq
P\times R_{2}.$Thus $A_{1}\left\langle s_{1}\right\rangle \subseteq P$ or $%
B_{1}\left\langle s_{1}\right\rangle \subseteq P.$ Hence $P$ is an almost
right $S_{1}$-prime ideal of $R_{1}$.
\end{proof}

\begin{remark}
Analogously of the proving of Theorem \ref{product}, the following equivalence holds: for any ideal $P$ of $R_{2}$, $P$ is an almost right $S$%
-prime ideal of $R_{2}$
if and only if $R_{1}\times P$ is an almost right $S$-prime ideal of $%
R_{1}\times R_{2}$.
\end{remark}
\begin{remark}
If $R$ is the direct product of $n$ rings $A_{1},A_{2},...,A_{n}$, and if $
S=\prod_{j=1}^{n}S_{j}$ is an $m$-system composed of the $m$-systems $%
S_{1},S_{2},...,S_{n}$.
\end{remark}

Then it follows that for some $k\in \{1,2,...,n\}$, if $I_{k}$ is an almost
right $S_{k}$-prime ideal
of $A_{k}$, then the ideal $A_{1}\times A_{2}\times ...\times I_{k}\times
...\times A_{n}$ is an almost right $S$-prime ideal of $R$. This result
establishes the connection between the almost right $S_{j}$-prime ideals in
the individual components $A_{j}$ and the almost right S-prime ideals in the
product ring $R$. It demonstrates how the properties of the component ideals
can be extended to construct a corresponding ideal in the direct product
ring.

\subsection{Idealization}

We now show how to construct almost right $S$-prime ideals using the Method of
Idealization. In what follows, $R$ is a ring (associative, not necessarily
commutative and not necessarily with identity) and $M$ is an $R-R$-bimodule.
The idealization of $M$ is the ring $R\boxplus M$ with $(R\boxplus
M,+)=(R,+)\oplus (M,+)$ and the multiplication is given by $%
(r,m)(s,n)=(rs,rn+ms).$ $R\boxplus M$ itself is, in a canonical way, an $R-R$%
-bimodule and $M\simeq 0\boxplus M$ is a nilpotent ideal of $R\boxplus M$ of
index 2. We also have $R\simeq R\boxplus 0$ and the latter is a subring of $%
R\boxplus M$. Note also that $R\boxplus M$ is a subring of the Morita ring $%
\left[ 
\begin{array}{cc}
R & M \\ 
0 & R%
\end{array}%
\right] $ via the mapping $(r,m)\mapsto \left[ 
\begin{array}{cc}
r & m \\ 
0 & r%
\end{array}%
\right] $. We will require some knowledge about the ideal structure of $%
R\boxplus M$. If $I$ is an ideal of $R$ and $N$ is an $R-R$- bi-submodule of 
$M$, then $I\boxplus N$ is an ideal of $R\boxplus M$ if and only if $%
IM+MI\subseteq N$. It follows from \cite{veldsman} that the prime ideals of $%
R\boxplus M$ are exactly the ideals of the form $I\boxplus M$ where$\ I$ is
a prime ideal of $R$.

It is clear that if $S$ is an m-system of $R$, then $S\boxplus
M=\{(s,k)|s\in S$ and $k\in M\}$ is a m-system of $R\boxplus M$.\bigskip
Note that $(S\boxplus M)\cap (I\boxplus M)=\emptyset $ if and only if $S\cap
I=\emptyset $.

\begin{theorem}\label{idealization}
Let $R$ be a ring with the proper ideal $I$ and $M$ an $R-R$-bi-module. If $%
S $ is an m-system of $R$, then $I\boxplus M$ is an almost right $S$ $%
\boxplus M$- prime ideal if and only if $I$ is an almost right $S$-prime
ideal of $R.$
\end{theorem}

\begin{proof}
$\Rightarrow $ Suppose $I\boxplus M$ is an almost right $S$ prime ideal of $%
R\boxplus M$. Let $aRb\subseteq I$ and $aRb\nsubseteq I^{2}$ where $a,b\in
R. $ Now $(a,0)R\boxplus M(b,0)=(aRb,aMb)\subseteq I\boxplus M$ and $%
(a,0)R\boxplus M(b,0)\nsubseteq I^{2}\boxplus M$. Hence $(a,0)R\boxplus
M(b,0)\nsubseteq \left( I\boxplus M\right) ^{2}.$ $I\boxplus M$ an almost
right $S$-prime ideal gives $(a,0)\left\langle (s,m)\right\rangle \subseteq
I\boxplus M$ or $(b,0)\left\langle \left( s,m\right) \right\rangle \subseteq
I\boxplus M$ for some $\left( s,m\right) \in $ $S$ $\boxplus M.$ Hence $%
(a,0)\left\langle (s,m)\right\rangle =(a,0)(R\boxplus M\left( s,m\right)
R\boxplus M=\left( RsR,RsM+RmR+MsR\right) \subseteq
(a,0)(RsR,M)=(a\left\langle s\right\rangle ,M)$ or $(a,0)\left\langle
s,m\right\rangle \subseteq (b\left\langle s\right\rangle ,M).$ Hence $%
a\left\langle s\right\rangle \subseteq I$ or $b\left\langle s\right\rangle
\subseteq I.$ So $I$ is almost right $S$-prime.

$\Leftarrow $Now suppose $I$ is an almost right $S$ prime ideal of $R.$ Let $%
(a,m_{1});(b,m_{2})\in R\boxplus M$ such that $(a,m_{1})R\boxplus
M(b,m_{2})=(aR,m_{1}R+aM)(b,m_{2})=(aRb,aRm_{2}+m_{1}Rb+aMb)\subseteq
I\boxplus M$ and $(a,m_{1})R\boxplus M(b,m_{2})\nsubseteq \left( I\boxplus
M\right) ^{2}.$ Hence $(a,m_{1})R\boxplus M(b,m_{2})\subseteq I\boxplus M$
and $(a,m_{1})R\boxplus M(b,m_{2})\nsubseteq I^{2}\boxplus M$. Thus $%
aRb\subseteq I$ and $aRb\nsubseteq I^{2}.$ Since $I$ is almost right $S$-prime, there exists $s\in S$ such that we have $a\left\langle s\right\rangle
\subseteq I$ or $b\left\langle s\right\rangle \subseteq I.$ Hence $%
(a,m_{1})\left\langle \left( s,m\right) \right\rangle =$ $(a,m_{1})R\boxplus
M(s,m)R\boxplus M)=(aR,aM+m_{1}R)(sR,sM+mR)\subseteq (aRsR,M)\subseteq
I\boxplus M$ or $(b,m_{2})(\left\langle s,m\right\rangle )\subseteq
(bRsR,M)=(b\left\langle s\right\rangle ,M)\subseteq I\boxplus M.$ Hence $%
I\boxplus M$ is almost right $S\boxplus M-$prime.
\end{proof}
The following theorem, is consequences of Theorem \ref{idealization}, when $S=\{1\}$.

\begin{theorem}
Let $R$ be a ring with the proper ideal $I$ and $M$ an $R-R$-bi-module. Then $I\boxplus M$ is an almost  prime ideal if and only if $I$ is an almost prime
ideal of $R.$
\end{theorem}

\end{document}